\documentclass[conference,a4paper]{IEEEtran}
\usepackage{amssymb}
\usepackage{amsmath}
\usepackage{amsthm}
\usepackage{subfigure, epsfig}
\usepackage{pstricks}
\usepackage{pstricks-add}
\usepackage{pst-plot}
\usepackage{pseudocode}
\usepackage[ruled,boxed,commentsnumbered]{algorithm2e}
\usepackage{graphicx}
\usepackage{ctable}
\usepackage{bbm}
\usepackage{rotating}
\usepackage{color}

\newtheorem{theorem}{Theorem}

\newtheorem{proposition}[theorem]{Proposition}

\newcommand{\mc}{\mathcal}

\newcommand{\ds}{\displaystyle}

\begin{document}

\sloppy
\title{Implicit Coordination in Two-Agent Team Problems; Application to Distributed Power Allocation}
\author{
   \IEEEauthorblockN{Benjamin Larrousse, Achal Agrawal, and Samson Lasaulce}
   \IEEEauthorblockA{L2S (CNRS -- Sup\'{e}lec -- Univ. Paris Sud 11),~91192 Gif-sur-Yvette, France\\
     Email: \{larrousse, agrawal, lasaulce\}@lss.supelec.fr}
}
\maketitle

\begin{abstract}
The central result of this paper is the analysis of an optimization problem which allows one to assess the limiting performance  of a team of two agents who coordinate their actions. One agent is fully informed about the past and future realizations of a random state which affects the common payoff of the agents whereas the other agent has no knowledge about the state. The informed agent can exchange his knowledge with the other agent only through his actions. This result is applied to the problem of distributed power allocation in a two-transmitter $M-$band interference channel, $M\geq 1$, in which the transmitters (who are the agents) want to maximize the sum-rate under the single-user decoding assumption at the two receivers; in such a new setting, the random state is given by the global channel state and the sequence of power vectors used by the informed transmitter is a code which conveys information about the channel to the other transmitter.
\end{abstract}

\section{Introduction and case study of interest}
\label{sec:intro}

Consider two agents or decision-makers who interact over a time period composed of a large number of stages or time-slots. At each stage $t \in \mathbb{N}$, agent $i\in\{1,2\}$ chooses an action $x_i \in \mathcal{X}_i$, $|\mathcal{X}_i| < +\infty$. The resulting (instantaneous) agents' common payoff is $u(x_0, x_1, x_2)$ where $x_0 \in \mathcal{X}_0$, $|\mathcal{X}_0| < +\infty$, is the realization of a random state for the considered stage, and $w$ is a real valued function. This random state is assumed to be an i.i.d. random process. To assess the (theoretical) limiting achievable coordination performance, it is assumed that one agent, agent 1, knows beforehand and perfectly all the realizations of the random state. Note that, in practice, the sole knowledge of the next realization of the random state is already very useful, just as in conventional power control problems. On the other hand, agent 2 does not know the state at all and can only be informed about it by observing the actions of agent 1. The performance analysis of this problem leads to deriving an information-theoretic constraint. The case of perfect observation is treated in \cite{Gossner-2006} while the generalization to noisy observations is conducted in \cite{Larrousse-isit2013}; to be precise, both references assume that agent 2 has a strictly causal knowledge of the state but it can be shown that not having any knowledge  about the state's realizations at all induces no limiting performance loss \cite{Khayutin-wcgts2007}. Reference \cite{Larrousse-isit2013} also states an optimization problem which essentially amounts to maximizing the long term payoff function under some constraints but this optimization problem is not analyzed. One of the purposes of this paper is precisely to study this general problem in detail. This will allow one to specialize it for the specific problem of power allocation in an important setting of cognitive radio.

The application of interest in this paper corresponds to a scenario which involves two transmitter-receiver pairs whose communications interfere each other. The communication system under consideration is modeled by an $M-$band interference channel, $M\geq1$, as depicted in Fig. \ref{fig:powerallocation}. In contrast with the vast majority of related works on distributed power allocation over multi-band channels (starting with the pioneering work \cite{Yu-2002}), the set of power allocation vectors at a transmitter is assumed to be discrete and finite (namely, $|\mathcal{X}_i| < + \infty$) instead of being continuous. This choice is motivated by many applications (see e.g., \cite{xing-ton-2008}\cite{belmega-asilomar-2010}\cite{rose-commag-2011}\cite{lasaulce-book-2011}) and by well-known results in information theory \cite{Cover:2006:EIT:1146355} which show that the continuous case generally follows from the discrete case by calling quantization arguments. We also assume that channel gains, as defined by Fig. \ref{fig:powerallocation}, lie in discrete sets; this is also well motivated by practical applications such as cellular systems in which quantities such as the channel quality indicator are used. Therefore, for the considered case study, $x_0$ is given by the vector of all channel gains $g_{ij}^m$, $(i,j)\in\{1,2\}^2, m \in \{1,2,...,M\}$, and lies in a finite discrete set (denoted by $\mathcal{X}_0$).

The paper is organized as follows. In Sec. \ref{sec:optimization-problem}, we introduce and solve the general optimization problem of interest. In Sec. \ref{sec:case-study}, we apply the general result of Sec. \ref{sec:optimization-problem} to a special case of payoff function and action sets for the agents. This special case corresponds to the problem of power allocation in a cognitive radio scenario. Sec. \ref{conclusion} concludes the paper.

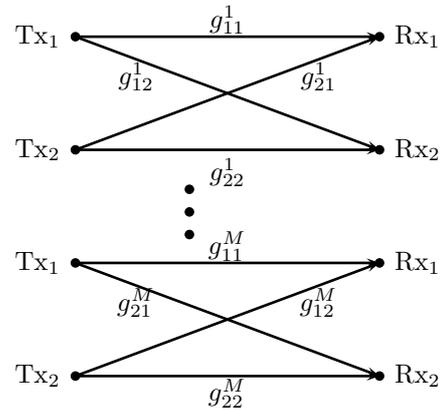
\begin{figure}[!ht]
\begin{center}
\psset{xunit=1cm,yunit=0.75cm}
\begin{pspicture}(0,-0.5)(4,6.5)
\psdot(0,0)
\psdot(0,2)
\psdot(4,0)
\psdot(4,2)
\psdot(0,4)
\psdot(0,6)
\psdot(4,4)
\psdot(4,6)
\psline[linewidth=1pt]{->}(0,0)(4,0)
\psline[linewidth=1pt]{->}(0,2)(4,2)
\psline[linewidth=1pt]{->}(0,4)(4,4)
\psline[linewidth=1pt]{->}(0,6)(4,6)
\psline[linewidth=1pt]{->}(0,0)(4,2)
\psline[linewidth=1pt]{->}(0,2)(4,0)
\psline[linewidth=1pt]{->}(0,4)(4,6)
\psline[linewidth=1pt]{->}(0,6)(4,4)
\rput[r](-0.2,0){$\mathrm{Tx}_2$}
\rput[r](-0.2,2){$\mathrm{Tx}_1$}
\rput[r](-0.2,4){$\mathrm{Tx}_2$}
\rput[r](-0.2,6){$\mathrm{Tx}_1$}
\rput[l](4.2,0){$\mathrm{Rx}_2$}
\rput[l](4.2,2){$\mathrm{Rx}_1$}
\rput[l](4.2,4){$\mathrm{Rx}_2$}
\rput[l](4.2,6){$\mathrm{Rx}_1$}
\rput[u](2,2.3){$g_{11}^M$}
\rput[u](2,6.3){$g_{11}^1$}
\rput[d](2,-0.3){$g_{22}^M$}
\rput[d](2,3.6){$g_{22}^1$}
\rput[u](3.2,1.3){$g_{12}^M$}
\rput[u](0.8,1.3){$g_{21}^M$}
\rput[u](3.2,5.3){$g_{21}^1$}
\rput[u](0.8,5.3){$g_{12}^1$}
\psdots*(1.5,2.5)(1.5,2.9)(1.5,3.3)
\end{pspicture}
\caption{Case study considered in Sec. \ref{sec:case-study}: an interference channel with 2 transmitters (Txs), 2 receivers (Rxs), and $M\geq1$ non-overlapping frequency bands. One feature of the retained model is that both power allocation policies and channel gains $g_{ij}^m$ are assumed to lie in finite discrete sets.}\label{fig:powerallocation}
\end{center}
\end{figure}

\section{Optimization problem analysis}\label{sec:optimization-problem}

Evaluating the limiting performance in terms of average payoff amounts to characterizing the set of possible values for $\mathbb{E}[u]$ under the information structure assumed in this paper. The expected payoff can be written as: \begin{align}
\mathbb{E}[u] &= \sum_{(x_0, x_1, x_2) \in \mathcal{X}_0 \times
\mathcal{X}_1 \times \mathcal{X}_2} q(x_0, x_1, x_2) u(x_0,x_1,x_2) \\
&= \sum_{(x_0, x_1, x_2) \in \mathcal{X}_0 \times
\mathcal{X}_1 \times \mathcal{X}_2} q_{X_0}(x_0)q(x_1, x_2|x_0) u(x_0,x_1,x_2)
\end{align}
where $q \in \Delta(\mathcal{X}_0 \times \mathcal{X}_1 \times \mathcal{X}_2)$,  $\Delta(\cdot)$ standing for the unit simplex over the set under consideration, $q_{X_0}$ is the marginal law of the random state and is considered fixed. The fact that the two agents can only coordinate through the considered information structure imposes a constraint on the average performance which translates into a constraint on $q$ for the expected payoff \cite{Gossner-2006}\cite{Larrousse-isit2013}:
\begin{equation}\label{eq:info-constraint}
I_{q}(X_0;X_2) - H_{q}(X_1|X_0,X_2) \leq 0
\end{equation}
where, for any two random variables $(X,Y) \in (\mc{X} \times \mc{Y})$ with joint law $q(\cdot,\cdot)$:
\begin{itemize}
\item $H_q(X|Y)$ is the conditional entropy of $X$ given $Y$ defined by:\\
\begin{equation} H_q(X|Y) = - \sum_{x \in \mc{X}} \sum_{y \in \mc{Y}} q(x,y) \log_2 \frac{q(x,y)}{q_Y(y)} \end{equation} 
 where $q_Y(\cdot)$ is obtained by marginalization of the joint distribution $q(\cdot,\cdot)$;\\
 One can note that the entropy of $X$ is simply:
 \begin{equation} H_q(X) = - \sum_{x \in \mc{X}} q_X(x) \log_2 q_X(x) \end{equation} 
 \item $I_q(X;Y)$ denotes the mutual information between $X$ and $Y$, defined by:
 \begin{equation} I_q(X;Y) = - \sum_{x \in \mc{X}} \sum_{y \in \mc{Y}} q(x,y) \log_2 \frac{q(x,y)}{q_X(x)q_Y(y)} \end{equation}
 \end{itemize}
Reference \cite{Larrousse-isit2013} provides a clear interpretation of this constraint. Essentially, the first term can be seen as a rate-distortion term while the second term can be seen as a limitation in terms of communication medium capacity. To state the optimization problem which characterizes the limiting performance in terms of expected payoff, a few notations are in order. We denote the cardinality of the set $\mathcal{X}_i$, $i\in\{0,1,2\}$ as: $|\mathcal{X}_i|=n_i<\infty$. For the sake of simplicity and without loss of generality, we consider $\mathcal{X}_i$ as a set of indices $\mathcal{X}_i= \{1, ..., n_i\}$. Additionally, we introduce the vector of payoffs (associated with the function $u$ defined earlier) $w = (w_1, w_2, ..., w_n) \in \mathbb{R}^n$ with $n = n_0 n_1 n_2$ and assume, without loss of generality, that $\mathrm{Pr}[X_0 = j] = \alpha_j > 0$  for all $j \in \mathcal{X}_0 = \{1,\dots,n_0 \}$, with $\sum_{j=1}^{n_0} \alpha_j = 1$. The indexation of $w$ and therefore the vector $q=(q_1,q_2,...,q_n)$ is chosen according to a lexicographic order. This is illustrated through Tab. \ref{Tab:indexing}. This choice simplifies the analysis of the optimization problem which is stated next.

\begin{table}[h]
\begin{center}
\begin{tabular}{|l||c|c|r|}
\hline
Index (i) & $X_0$ & $X_1$ & $X_2$\\
\specialrule{.3em}{.2em}{.2em}
1 & 1&1&1\\
2 & 1 & 1 & 2\\
\vdots & \vdots & \vdots & \vdots \\
$n_2$ & 1 & 1 & $n_2$\\
\hline
$n_2 + 1$ & 1 & 2 & 1\\
\vdots & \vdots & \vdots & \vdots \\
$2 n_2$ & 1 & 2 & $n_2$\\
\hline
\vdots & \vdots & \vdots & \vdots \\
\hline
$n_2(n_1 - 1) + 1$ & 1 & $n_1$ & 1\\
\vdots & \vdots & \vdots & \vdots \\
$n_1 n_2$ & 1 & $n_1$ &$n_2$\\
\specialrule{.3em}{.2em}{.2em}
\vdots & \vdots & \vdots & \vdots \\
\vdots & \vdots & \vdots & \vdots \\
\specialrule{.3em}{.2em}{.2em}
$n_1 n_2 (n_0 - 1) + 1$ & $n_0$ & 1 & 1\\
\vdots & \vdots & \vdots & \vdots \\
$n_0 n_1 n_2$ & $n_0$ & $n_1$ & $n_2$ \\
\specialrule{.3em}{.2em}{.2em}
\end{tabular}
\caption{Chosen indexation for the payoff vector $w$ and distribution vector $q$. Bold lines delineate blocks of size $n_1 n_2$ and each block corresponds to a given value of the random state $X_0$.}\label{Tab:indexing}
\end{center}
\end{table}

The Information Constraint \eqref{eq:info-constraint} can be re-written as:
\begin{align}
ic(q) &\stackrel{\vartriangle}{=} I_{q}(X_0;X_2) - H_{q}(X_1|X_0,X_2) \\
&= H_{q}(X_0) + H_{q}(X_2) - H_{q}(X_0,X_1,X_2) 
\end{align}
With our notation, we have:
\begin{align}
H_{q}(X_0) = - \displaystyle{\sum_{i=1}^{n_0}} \left[ (\displaystyle{\sum_{j=1+(i-1)n_1n_2}^{in_1n_2}} q_j)\log_2(\displaystyle{\sum_{j=1+(i-1)n_1n_2}^{i n_1n_2}} q_j) \right] \label{eq:HX0q}
\end{align}
\begin{align}
H_{q}(X_2) = - \displaystyle{\sum_{i=1}^{n_2}} \left[ (\displaystyle{\sum_{j=0}^{n_0n_1-1}} q_{i+jn_2})\log_2(\displaystyle{\sum_{j=0}^{n_0n_1-1}} q_{i+jn_2}) \right] \label{eq:HX2q}
\end{align}
and
\begin{align}
- H_{q}(X_0,X_1,X_2) = \displaystyle{\sum_{i=1}^{n_0n_1n_2}} q_i\log_2q_i \label{eq:HX0X1X2q}
\end{align}

Thus, the optimization problem of interest consists of finding the best joint distribution(s) $q$ (i.e., the best correlation between the agent's actions and the random state) and is as follows:
\begin{equation}
\begin{array}{cl}
\min & -\mathbb{E}_{q}[w]  =  - \displaystyle{\sum_{i=1}^{n_0n_1n_2}} q_i w_i \\
\text{s.t.} & \displaystyle{-1 + \sum_{i=1}^{n_0n_1n_2}} q_i = 0 \\
  & \displaystyle{ - \alpha_i+\sum_{j=1+(i-1)n_1n_2}^{in_1n_2}} q_j   =  0, \qquad \forall i \in \{1,\dots,n_0\} \\
    &     -q_i \leq  0,\phantom{=========} \forall i \in \{ 1,2,\dots,n_0n_1n_2\} \\
 & ic(q) \leq 0
\end{array} \label{Optpb}
\end{equation}

The first and third constraints imposes that $q$ has to be a probability  distribution. The second constraint imposes that the marginal of $q$ with respect to $x_1$ and $x_2$ has to coincide with the distribution of the random state which is fixed. The fourth constraint is the information-theoretic constraint \eqref{eq:info-constraint}.

To solve the optimization problem \eqref{Optpb} we will apply the Karush Kuhn Tucker (KKT) necessary conditions for optimality \cite{boyd-book-2004}. For this purpose, we first verify that strong duality holds. This can be done e.g., by proving that Slater's constraint qualification conditions are met. Namely, there exists a strictly feasible point for (\ref{Optpb}) and that (\ref{Optpb}) is a convex problem. First, by specializing Lemma 1 in \cite{Larrousse-isit2013} in the case of perfect observation, we know that (\ref{eq:info-constraint}) defines a convex set. Since the cost function and the other constraints of the problem are affine, the problem is then convex; as a consequence. KKT conditions are also sufficient for optimality. The existence of a feasible point is stated in the next proposition.

\begin{proposition}
There exists a strictly feasible distribution $q^+ \in \Delta(\mathcal{X}_0 \times \mathcal{X}_1 \times \mathcal{X}_2)$ for the optimization problem \eqref{Optpb}.
\end{proposition}

\begin{proof}
First, choose a triplet of random variables $(X_0, X_1,X_2)$ which are independent. That is, we consider a joint distribution $q^+$ which is of the form $q^+(x_0,x_1,x_2) = q_{X_0}^+(x_0)q_{X_1}^+(x_1)q_{X_2}^+(x_2)$. Second, one can always impose a full support condition to the marginals $q_{X_1}^+$ and $q_{X_2}^+$ (i.e., $\forall x_i, q_{X_i}^+(x_i) >0$); $q_{X_0}^+ \equiv q_{X_0}$ has a full support by assumption. Therefore, for the distribution $q^+(x_0,x_1,x_2)$ to be strictly feasible, it remains to be checked that the information-theoretic constraint is active. And this is indeed the case since:
\begin{align}
I_{q}(X_0;X_2) - H_{q}(X_1|X_0,X_2) &= 0 - H_{q}(X_1|X_0,X_2) \label{Indep1} \\
&= - H_{q}(X_1) \label{Indep2}\\
&< 0 \label{positivness}
\end{align}
where: \eqref{Indep1} and \eqref{Indep2} come from the independence hypothesis between $X_0$, $X_1$, and $X_2$; \eqref{positivness} comes from the positiveness of  the entropy and the fact that every $q^+(x_0,x_1,x_2)$ (and thus every $q_{X_1}^+(x_1)$) is strictly positive.
\end{proof}

Following the previous considerations, KKT conditions can be applied. The Lagrangian function can be written as:
\begin{align}
 &\mathcal{L}(q,\mu, \mu_0, \lambda, \lambda_{\mathrm{IC}}) = - \sum_{i=1}^{n_0n_1n_2} (w_i q_i + \lambda_i q_i) \nonumber \\
&+ \mu_0\left[\sum_{i=1}^{n_0n_1n_2} q_i-1\right] + \sum_{i=1}^{n_0} \mu_{i}\left[\sum_{j=1+(i-1)n_1n_2}^{i n_1 n_2} q_j - \alpha_i \right] \nonumber \\
&+ \lambda_{\mathrm{IC}}\cdot ic(q)
\end{align}
where $\lambda = (\lambda_1, ..., \lambda_{n_0 n_1 n_2})$, $\mu = (\mu_1, ..., \mu_{n_0})$, and IC stands for information-theoretic constraint. We have the following partial derivatives for the information constraint:
\begin{align}
&\frac{\partial ic}{\partial q_i}(q) = \Bigg[-\sum_{k=1}^{n_0} \mathbbm{1}_{\{1+(k-1)n_1n_2 \leq i \leq (k)n_1n_2\} } \nonumber \\
& \phantom{===========} * \log_2(\sum_{j=1+(k-1)n_1n_2}^{k n_1n_2} q_j)\nonumber \\
& -\sum_{k=1}^{n_2} \mathbbm{1}_{\{ i \in \{k,k+n_2,\dots,k+(n_0n_1-1)n_2 \} \} } \log_2(\sum_{j=0}^{n_0n_1-1} q_{k+jn_2}) \nonumber \\
&+ \log_2 q_i - 1 \Bigg] \qquad \forall i \in \{ 1,2,\dots,n_0n_1n_2\}
\end{align}

Other terms of the Lagrangian are easy to derive. KKT conditions follow:
\begin{align}
&\frac{\partial \mathcal{L}}{\partial q_i} = -w_i - \lambda_i + \mu_0 + \sum_{j=1}^{n_0} \mu_{j} \mathbbm{1}_{ \{1+n_1 n_2(j-1) \leq i \leq j n_1 n_2\} } \nonumber \\
& + \lambda_{\mathrm{IC}} \Bigg[\frac{\partial ic}{\partial q_i}(q)\Bigg] = 0 \qquad \forall \; i \in\{1,2,\dots,n_0n_1n_2\} \label{partialLAGg} \\
&\lambda_i \geq 0 \qquad \forall \; i \in\{1,2,\dots,n_0n_1n_2\} \\
&\lambda_{\mathrm{IC}} \geq 0 \\
&\lambda_i q_i = 0 \qquad \forall \; i \in\{1,2,\dots,n_0n_1n_2\} \\
&\lambda_{\mathrm{IC}} ic(q) = 0
\end{align}
where $\mathbbm{1}_{\{\cdot\}}$ is the indicator function and $i(q)$ is the inequality constraint function associated with the information-theoretic constraint (\ref{eq:info-constraint}). By inspecting the KKT conditions, the following proposition can be proved.

\begin{proposition} If there exists a permutation such that the payoff vector $w$ can be strictly ordered, then any optimal solution of (\ref{Optpb}) is such that the information-theoretic constraint is active i.e.,  $\lambda_{\mathrm{IC}} > 0$.
\end{proposition}

\begin{proof} We proceed by contradiction. Assume that the payoff vector can be strictly ordered and that the constraint is not active for solutions under consideration, that is, $\lambda_{\mathrm{IC}} = 0$.

First, consider possible solution candidates $q$ which have two or more non-zero components per block of size $n_1 n_2$ which is associated with a given realization $x_0$ of the random state (see Tab. \ref{Tab:indexing}) . Since there exists a pair of distinct indices $(j,k)$ such that $q_j>0$, $q_k>0$, we have that $\lambda_j=0$, $\lambda_k = 0$. This implies that, through the gradient conditions of the KKT conditions, $w_j = w_k$ which contradicts the fact that payoffs are strictly ordered.

Second, consider possible solution candidates $q$ which have only one non-zero component per block associated with $x_0$ (see Tab. \ref{Tab:indexing}). This implies that
$ H_q(X_0, X_1, X_2) = H_q(X_0) = H(X_0)$, which means that $H_q(X_0) + H_q(X_2) > H_q(X_0, X_1, X_2)$, whenever $H_q(X_2) >0$. This means that the constraint is violated and therefore the considered candidates are not feasible. Now, if $H_q(X_2) =0$, we see that the Information constraint is active which contradicts again the starting assumption .
\end{proof}

Proposition 2 is especially useful for wireless communications when the state is given by the overall channel. Due to channel randomness, the most common scenario is that the payoffs associated with the channel realizations are distinct. For this reason, we will assume such a setting in this paper and thus that $\lambda_{\mathrm{IC}} > 0$. If $\lambda_{IC} > 0$, we have the following: \begin{itemize}
    \item We can not have $\lambda_i>0$ for one or more $i \in \{1,2,\dots,n_0n_1n_2\}$. Indeed, if for example $\lambda_i >0$, then $q_i = 0$, which implies $\log_2 (q_i) = - \infty$ and  \eqref{partialLAGg} can not be satisfied.
	\item However, if one of the $q_i$'s equals $0$, and $q_{k}=0$ for all $k$ such that $k[n_2] = i[n_2]$ (where $[x]$ stands for modulo $x$), then the $\lambda_{\mathrm{IC}}$ component equals $\lim_{x \to 0}\frac{x}{n_0n_1x}$ and does not go to $-\infty$. This case cannot be discarded, but it can be said that $X_2$ is deterministic in such a case.
\end{itemize}
Summarizing our analysis, the only possible cases are:\begin{itemize}
 \item $\lambda_{\mathrm{IC}} > 0$, and exactly one $\lambda_i$ for each block (corresponding to a particular state of nature) are non-zeros, and they have to be associated with the same action of $X_2$ ($X_2$ has to be deterministic). In this case there is no communication, and the optimal strategies are trivial. Therefore we shall not be discussing this case henceforth.
 \item The only relevant case is: \begin{align*}
\lambda_i &= 0 \qquad \forall \; i \in \{1,2,\dots,{n_0n_1n_2}\} \\
\lambda_{\mathrm{IC}} &> 0
\end{align*}
\end{itemize}

For the latter case, KKT conditions become:
\begin{align}
&\frac{\partial \mathcal{L}}{\partial q_i} = -w_i + \mu_0 + \left( \sum_{j=1}^{n_0} \mu_{j} \mathbbm{1}_{ \{1+n_1 n_2(j-1) \leq i \leq j n_1 n_2\} } \right) \nonumber \\
& + \lambda_{\mathrm{IC}} \Bigg[\frac{\partial ic}{\partial q_i}(q) \Bigg] = 0 \qquad \forall \; i \in\{1,2,\dots,n_0n_1n_2\} \\
&\lambda_i = 0 \qquad \forall \; i \in\{1,2,\dots,{n_0n_1n_2}\} \\
&\lambda_{\mathrm{IC}} > 0 \\
&i(q) = 0.
\end{align}

Now that we have proved some useful results about the structure of optimal solutions of (\ref{Optpb}), a natural question is whether the optimal solution is unique, which is the purpose of the next proposition.

\begin{proposition} If there exists a permutation such that the payoff vector $w$ can be strictly ordered, the optimization problem (\ref{Optpb}) has a unique solution.
\end{proposition}

\begin{proof} We know, by Prop. 2, that $\lambda_{\mathrm{IC}} >0$ for any optimal solution. It turns out that, if $\lambda_{\mathrm{IC}} >0$, the Lagrangian of (\ref{Optpb}) is a strictly convex function w.r.t. the vector $q$. Indeed, the optimization spaces are compact and convex, and the Lagrangian is the sum of linear functions and a strictly convex function $i(q)$.\\
It remains to show that $ic : q  \mapsto I_{q} (X_0 ; X_2) - H_{q} (X_1 |X_0,X_2)$ is strictly convex over the set of distributions $q \in \Delta(\mathcal{X}_0 \times \mathcal{X}_1 \times \mathcal{X}_2)$ that verify $q_{X_0} := \sum_{(x_1,x_2)} q(x_0,x_1,x_2) = \rho(x_0)$ with $\rho$ fixed.

The first term $I_{q} (X_0 ; X_2)$ is a convex function of $q_{X_2|X_0}$ for fixed $q_{X_0}$. For the second term, let $\lambda_1 \in [0,1]$, $\lambda_2 = 1 - \lambda_1$, $(q^1,q^2) \in  (\Delta(\mathcal{X}_0 \times \mathcal{X}_1 \times \mathcal{X}_2))^2$ and $q=\lambda_1 q^1 + \lambda_2 q^2$. We have that:
\begin{align}
  & H_q(X_1 | X_0, X_2) =  - \sum_{x_0,x_1,x_2} \bigg( \sum_{i=1}^{2} \lambda_i q^i (x_0,x_1,x_2) \bigg). \nonumber \\
  &\phantom{==============} \log \left[ \frac{\sum_{i=1}^2 \lambda_i q^i(x_0,x_1,x_2)}{ \sum_{i=1}^2 \lambda_i q_{X_2}^{i}(x_2)} \right] \\
  &> -\sum_{x_0,x_1,x_2}  \sum_{i=1}^2 \lambda_i q^i(x_0,x_1,x_2) \log \left[ \frac{\lambda_i q^i(x_0,x_1,x_2)}{\lambda_i q_{X_2}^{i}(x_2)} \right] \label{Ineqlogsum}\\
&=  -  \sum_{i=1}^2 \lambda_i \sum_{x_0,x_1,x_2} q^i(x_0,x_1,x_2) \log \left[ \frac{q^i(x_0,x_1,x_2)}{ q_{X_2}^{i}(x_2)} \right] \\
&= \lambda_1 H_{q^1}(X_1 | X_0, X_2) + \lambda_2 H_{q^2}(X_1 |X_0, X_2)
\end{align}
where \eqref{Ineqlogsum} comes from the log sum inequality \cite{Cover:2006:EIT:1146355}, with:
\begin{equation} a_i = \lambda_i q^i(x_0,x_1,x_2)
\end{equation}
and
\begin{equation} b_i = \lambda_i q_{X_2}^{i}(x_2)
\end{equation}
for $i=1,2$ and for all $x_0,x_1,x_2$ such that $q_{X_2}^{i}(x_2)>0$.

The inequality is strict because $\frac{a_1}{b_1} \neq \frac{a_2}{b_2}$, since we have assumed that $q^1$ and $q^2$ distinct.
\end{proof}

The uniqueness property for the optimization problem is particularly useful in practice since it means that any converging numerical procedure to find an optimal solution will lead to the unique global minimum.

\section{Distributed power allocation case study}\label{sec:case-study}

\subsection{Case study description}

We now consider the specific problem of power allocation over $M-$band interference channels with two transmitter-receiver pairs. Transmissions are time-slotted and, on each time-slot, transmitter $i\in \{1,2\}$ has to choose a power allocation vector in the following set of actions:
\begin{align}
\mathcal{P}_i  = \bigg\{ &\frac{{P}_{\max}}{\ell} e_{\ell} : \ell \in\{1, \hdots, M\}, \nonumber\\
& e_{\ell} \in \{0,1\}^{M}, \sum_{i =1}^{M} e_{\ell}(i) = \ell  \bigg\}
\end{align}
where $P_{\max}$ is the the power budget available at a transmitter. Each channel is assumed to lie in a discrete set $\Gamma=\{g_1, ...,g_S\}$, $S\geq 1$, $g_s \geq 0$ for $s\in\{1,...,S\}$. Therefore, if one denotes by $g^m$ the vector of four channel gains corresponding to the band $m \in \{1,...,M\}$, then $g^m \in \Gamma^4$ and the global channel state $g=[g^1,...,g^M]$ lies in $ \mc{G} = \Gamma^{4M}$ whose cardinality is $S^{4M}$. As it is always possible to find a one-to-one mapping between $\mc{P}_i$, $i\in\{1,2\}$, (resp. $\mc{G}$) and $\mc{X}_i$ (resp. $\mc{X}_0$) as defined in Sec. \ref{sec:optimization-problem}, the results derived therein can be applied here. Lastly, for a given time-slot, the instantaneous or the stage payoff function which is common to the transmitters is chosen to be:
\begin{align}
u:\left|
\begin{array}{c}
\mc{G} \times \mc{P}_1 \times \mc{P}_2  \rightarrow \mathbb{R}^+ \phantom{================} \\
(g,p_1,p_2)  \mapsto \ds{\sum_{i=1}^2} \ds{\sum_{m=1}^M} B_m \log_2 \left(1+  \frac{g_{ii}^m p_{i}^m}{\sigma^2 + g_{-ii}^m p_{-i}^m} \right)
\end{array}
\right.
\end{align}
where $p_i$ is the power allocation chosen by transmitter $i$ on the current time-slot whose channel state is $g$, $\sigma^2$ is the noise variance, $B_m$ is the bandwidth of band $m$, $p_i^m$ the power transmitter $i$ allocates to band $i$,$-i$ stands for the  transmitter other than $i$.

\subsection{Simulation setup}

In this section, specific values for the parameters which are defined in the preceding section are chosen, in particular to make the interpretations relatively easy. We assume $M=2$ bands and therefore that the transmitters have three actions: $\mathcal{P}_i= P_{\max}\left\{ (0,1),(1,0),(\frac{1}{2},\frac{1}{2}) \right\}$ for $i\in \{1,2\}$. As \cite{mochaourab-valuetools-2009} we assume the first band to be protected ($g_{12}^1=g_{21}^1=0$) whereas the second band corresponds to a general single-band interference channel. The other channel gains are chosen as follows:
\begin{equation}
g_{ii}^1 \in \{ 0.1, 1.9\}, \quad i\in\{1,2\}
\end{equation}
\begin{equation}
g_{ij}^2 \in \{ 0.15, 1.85\}, \quad (i,j)\in\{1,2\}.
\end{equation}
We suppose that each $g_{ij}^k$, $k=1,2$ is i.i.d. and Bernouilli distributed $g_{ij}^k \sim \mathcal{B}(\pi_{ij}^k)$ with $P(g_{ii}^1 = 0.1) = \pi_{ii}^1$ and  $P(g_{ij}^2 = 0.15) = \pi_{ij}^2$. We define SNR[dB]$= 10 \log_{10}\left( \frac{P_{\max}}{\sigma^2} \right)$, and we consider two regimes for the second band: a high interference regime (HIR), defined by $(\pi_{11}^2,\pi_{12}^2,\pi_{21}^2,\pi_{22}^2) = (0.5,0.1,0.1,0.5)$ and a low interference regime (LIR) defined by $(\pi_{11}^2,\pi_{12}^2,\pi_{21}^2,\pi_{22}^2) = (0.5,0.9,0.9,0.5)$. For the first band, we take $\pi_{11}^1=\pi_{22}^1=0.2$. One can see that our choice of parameters indeed define a high interference regime: $ P((g_{ij}^2|i\neq j) = 1.85) = 1-0.1 = 0.9$, thus creating high interference due a high probability for a greater value of $(g_{ij}^2|i\neq j)$. The similar intuition holds for low interference regime. Three power allocation policies will be considered:

\begin{itemize}
	\item The costless communication case, where both transmitters knows the state beforehand and can reach the maximum payoff at every stage;
	\item The (information-constrained) optimal policy (OP) corresponding to the optimal solution of the optimization problem \eqref{Optpb};
	\item The blind policy (BP), where transmitters don't know anything about channel gains and always choose to put half of their power in each band: $p_1=p_2= P_{\max} ( \frac{1}{2} , \frac{1}{2})$ at every stage.
\end{itemize}

Fig. \ref{fig:refhalfhalf} represents the gain allowed by asymmetric coordination w.r.t. the case where the transmitters always use the uniform power allocation policy (BP). This gain can be as high as $40\%$ for the considered range of SNR. It is seen that the gains are particularly significant when the interference is high (the two top red curves) and in the low and high SNR regimes (red and blue curves on the left and right sides). The first observation translates the intuition that the higher the interference level the stronger is the gain brought by coordination. The second can be understood as follows: In the high SNR regime, the transmission rate over the non-protected band is interference limited and bounded and it is better to allocate the power to the protected band which allows an arbitrarily large rate as the SNR grows large. This explains why allocating uniformly the power becomes more and more suboptimal as the SNR increases. In the low SNR regime, essentially the interference becomes negligible and the best power allocation policies roughly correspond to water-filling over the available channels. At low SNR, the best water-filling policy is to use the best band and not to allocate power uniformly, which explains the gap between the coordinated policies and uniform power allocation. Our explanations are sustained by Fig. \ref{fig:Marg-HIR}, which shows the probability that a transmitter uses a given power allocation vector. For instance, at low SNR, the dominant actions for both transmitters is to use the protected band. It can be noticed that transmitter 1 has also to convey information to transmitter 2 (i.e., ensuring that the entropy of $X_1$ is not too small), which is why he cannot use the protected band as often as transmitter 2. One also notices in Fig.  \ref{fig:Marg-HIR} that the probability of the action $(0,1)$ (using the shared band) is zero from lower SNR values for transmitter 2 than for transmitter 1. This can be explained by the fact that the higher the power available for both transmitters, the higher the interference in the non-protected band. However, transmitter 1 still chooses to play this action as it has knowledge of channel gains and can use the interference band to improve the common utility. The same argument stands for Fig. \ref{fig:Marg-beta}. At last, Fig. \ref{fig:Marg-beta} shows the influence of the bandwidths on the power allocation policies. Not surprisingly, the higher the bandwidth of the protected band, the more often it is used, and conversely for the non-protected band. Concerning the uniform policy, it is seen that transmitter uses it more frequently, although channel conditions are similar, which translates again the need for transmitter 1 to convey information.

\begin{figure}[h!]
\centering
\includegraphics[width=0.52\textwidth]{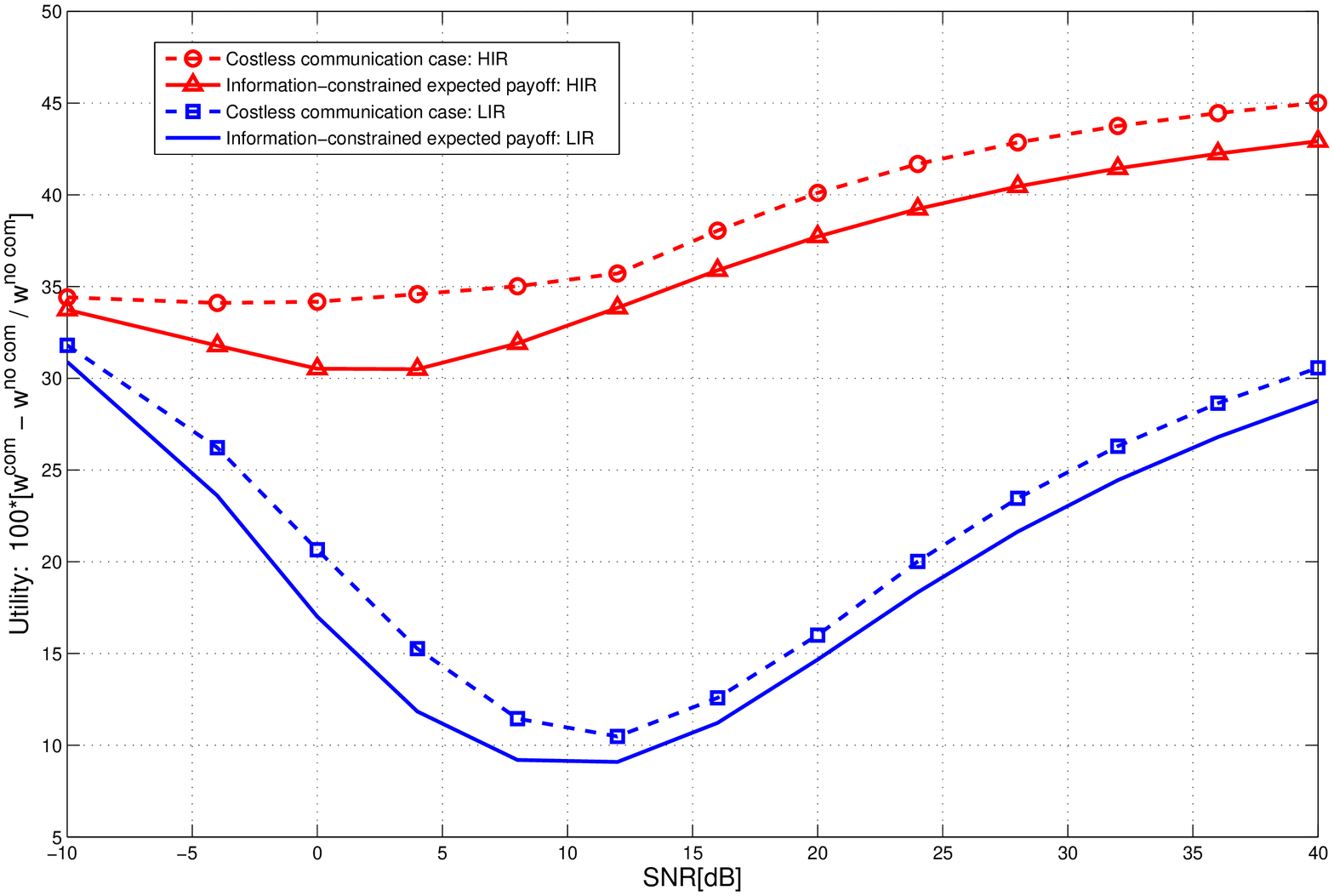}
\caption{Relative gain in terms of expected payoff (``OP/BP - 1'' in [\%]) vs SNR[dB] obtained with the Optimal policy (OP) (with and without communication cost) when the reference policy is to put half of the power on each band (BP). Red curves correspond to the HIR, and blue curve to the LIR. $B_1=B_2=10$MHz.}
\label{fig:refhalfhalf}
\end{figure}

\begin{figure}[htbp]
\centering
 \includegraphics[width=0.52\textwidth]{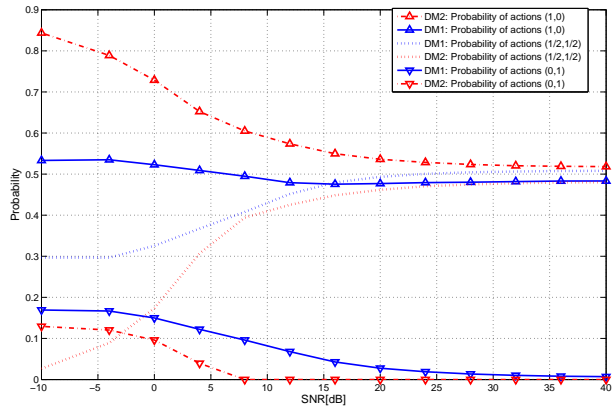}
\caption{Marginal probability distributions $q_{X_1}(\cdot)$ $q_{X_2}(\cdot)$ of transmitter 1 and transmitter 2 for the optimal policy vs SNR[dB] for the High Interference Regime. $B_1=B_2=10$MHz.}
\label{fig:Marg-HIR}
\end{figure}

\begin{figure}[htbp]
\centering
 \includegraphics[width=0.52\textwidth]{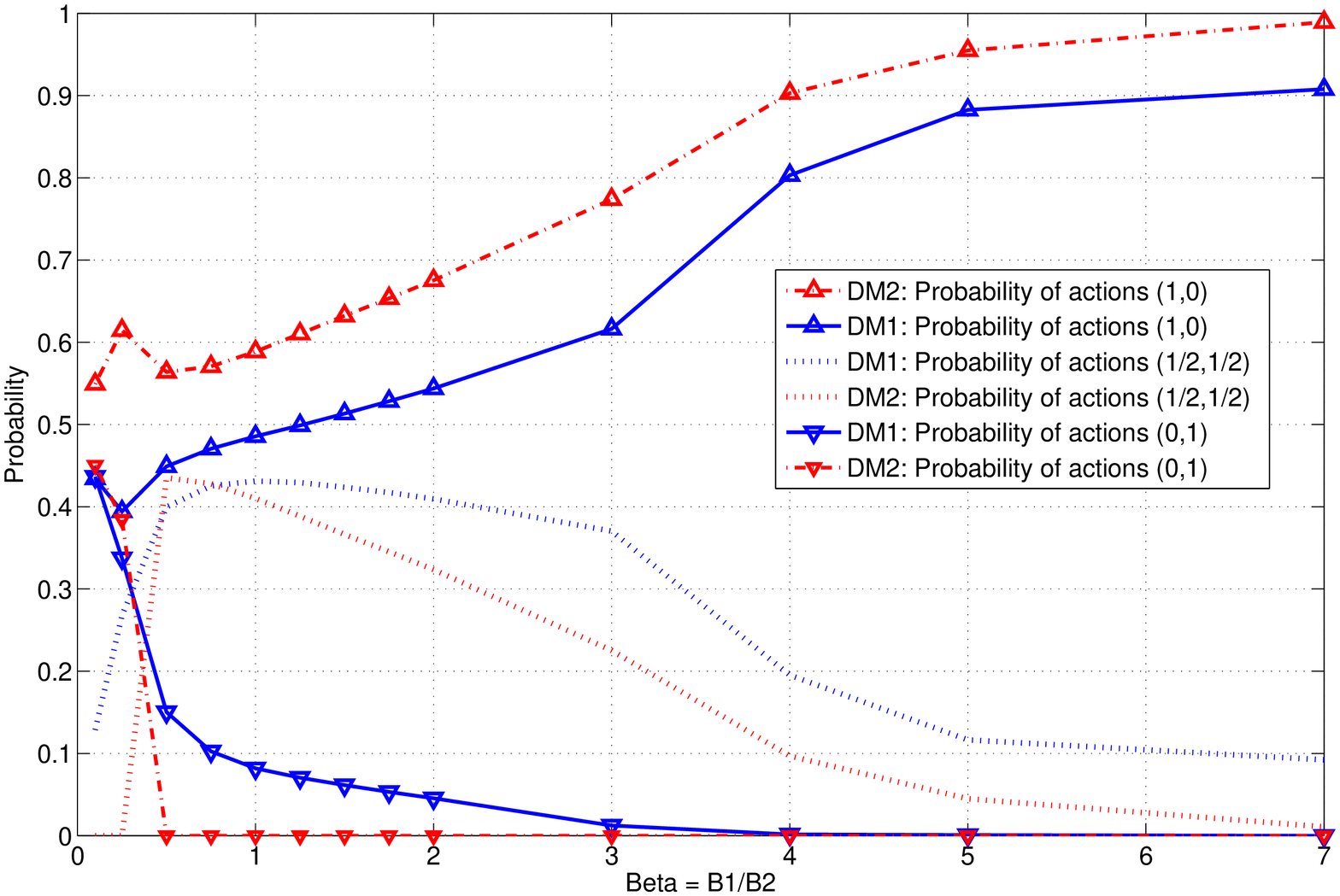}
\caption{Marginal probability distributions $q_{X_1}(\cdot)$ $q_{X_2}(\cdot)$ of transmitter 1 and transmitter 2 for the optimal policy vs $\beta$ for the optimal policy for the High Interference Regime, where $\beta= \frac{B_1}{B_2}$. For this simulation, we chose SNR=10[dB].}
\label{fig:Marg-beta}
\end{figure}

\section{Conclusion}\label{conclusion}

This work clearly illustrates the potential benefit of the proposed approach, by embedding coordination information into the power allocation levels, relative gains as high as $40\%$ can be obtained w.r.t. the uniform power allocation policies. In this work, the embedded information is a distorted version of the channel state but the proposed approach is much more general: information about the state of queue, a battery, etc, could be considered; other types of policies might be considered to encode information e.g., channel selection policies, transmit power levels. The study of generalized versions of this problem, such as the case of imperfect monitoring, or continuous power allocation, will be provided in future works. This work however indicates the high potential of such an approach for team optimization problems. More importantly, it gives an optimization framework to analyze performance limits for problems with implicit communication.

\section*{Acknowledgment}
Auhtors of this paper and particularly A. Agrawal would like to thank project LIMICOS - ANR-12-BS03-0005 for financing the project.

\bibliographystyle{IEEEtran}
\bibliography{biblio-wiopt2}

\end{document}